\documentclass[11pt]{amsart}
\usepackage{tikz-cd}
\usepackage{verbatim}
\usepackage[colorinlistoftodos]{todonotes}
\usepackage{mathtools}
\usepackage[bookmarks,colorlinks,breaklinks]{hyperref}  
\usepackage[top=2cm, bottom=4.5cm, left=3cm, right=3cm]{geometry}
\hypersetup{linkcolor=blue,citecolor=blue,filecolor=blue,urlcolor=blue} 
\usepackage{tikz}\usetikzlibrary{cd,fit,matrix,arrows,decorations.pathmorphing}
\tikzset{commutative diagrams/.cd}
\usepackage{bm}

\numberwithin{equation}{subsection}

\newtheorem{theorem}{Theorem}[section]
\newtheorem{corollary}[theorem]{Corollary}
\newtheorem{lemma}[theorem]{Lemma}
\newtheorem{proposition}[theorem]{Proposition}

\theoremstyle{definition}
\newtheorem{definition}[theorem]{Definition}
\newtheorem{definition-theorem}[theorem]{Definition-Theorem}

\newtheorem{claim}[theorem]{Claim}

\newtheorem{remark}[theorem]{Remark}

\theoremstyle{remark}

\newtheorem*{remark*}{Remark}

\usepackage{enumitem}
\setenumerate[1]{label=\textup{(\alph*)}}
\setenumerate[2]{label=\textup{(\roman*)}}

\newcommand{\Aberk}[1]{\ensuremath{\bA_{\mathrm{Berk},#1}^1}}
\newcommand{\g}{\bfc}
\newcommand{\supp}{{\rm Supp}}
\newcommand{\cD}{\mathcal{D}}
\newcommand{\E}{{\mathbb E}}
\renewcommand{\d}{\delta}
\renewcommand{\l}{\lambda}
\newcommand{\gal}{{\rm Gal}}
\newcommand{\trdeg}{{\rm trdeg}}

\newcommand{\bfc}{{\mathbf c}}
\newcommand{\N}{{\mathbb N}}
\newcommand{\M}{{\mathbb M}}

\newcommand{\C}{{\mathbb C}}
\newcommand{\A}{{\mathbb A}}
\newcommand{\bA}{{\mathbb A}}
\newcommand{\R}{{\mathbb R}}
\newcommand{\Fp}{{\mathbb F}_p}

\newcommand{\Fpbar}{{\overline{\mathbb F}_p}}

\newcommand{\Lbar}{{\overline{L}}}

\newcommand{\hhat}{\widehat{h}} 
\newcommand{\lra}{\longrightarrow}

\begin{document}
\title[Simultaneously preperiodic points]{Simultaneously preperiodic points for a family of polynomials in positive characteristic}

\author{Dragos Ghioca}
\address{Department of Mathematics \\ University of British Columbia \\ 1984 Mathematics Road \\ Canada V6T 1Z2}
\email{dghioca@math.ubc.ca}



\begin{abstract}
In the goundbreaking paper \cite{Matt} (which opened a wide avenue of research regarding unlikely intersections in arithmetic dynamics), Baker and DeMarco prove that for the family of polynomials $f_\lambda(x):=x^d+\lambda$ (parameterized by $\lambda\in\mathbb{C}$), given two starting points $a$ and $b$ in $\mathbb{C}$, if there exist infinitely many $\lambda\in\mathbb{C}$ such that both $a$ and $b$ are preperiodic under the action of $f_\lambda$, then $a^d=b^d$. In this paper we study the same question, this time working in a field of characteristic $p>0$. The answer in positive characteristic is more nuanced, as there are three distinct cases: (i) both starting points $a$ and $b$ live in $\Fpbar$; (ii) $d$ is a power of $p$; and (iii) not both $a$ and $b$ live in $\Fpbar$, while $d$ is not a power of $p$. Only in case~(iii), one derives the same conclusion as in characteristic $0$, i.e., that $a^d=b^d$. In case~(i), one has that for each $\lambda\in\Fpbar$, both $a$ and $b$ are preperiodic under the action of $f_\lambda$, while in case~(ii), one obtains that \emph{also} whenever $a-b\in\Fpbar$, then for each parameter $\lambda$, we have that $a$ is preperiodic under the action of $f_\lambda$ if and only if $b$ is preperiodic under the action of $f_\lambda$.    
\end{abstract}

\maketitle

\section{Introduction}
\label{sec:intro}

We start by setting up some basic notation for our paper in Subsection~\ref{subsec:notation}.


\subsection{Notation}
\label{subsec:notation}

Throughout this paper, given a self-map $f$ on some quasiprojective variety $X$, we denote by $f^n$ its $n$-th compositional power; by convention, $f^0$ represents the identity map ${\rm id}_X$ on $X$. A preperiodic point $x\in X$ for $f$ has the property that $f^m(x)=f^n(x)$ for some $0\le m<n$; if $m=0$ (i.e., $f^n(x)=x$), then the point $x$ is called periodic (under the action of $f$). 


\subsection{Our results}
We prove the following main result.

\begin{theorem}
\label{thm:main}
Let $d\ge 2$ be an integer, let $L$ be a field of characteristic $p>0$ and let $\alpha,\beta\in L$. We let $\Lbar$ be a fixed algebraic closure of $L$, and we let $\Fpbar$ be the algebraic closure of $\Fp$ inside $\Lbar$. We consider the family of polynomials 
$$f_\lambda(x):=x^d+\lambda\text{ parameterized by $\lambda\in \Lbar$.}$$ 
Then there exist infinitely many $\lambda\in \Lbar$ such that both $\alpha$ and $\beta$ are preperiodic under the action of $f_\lambda$ if and only if at least one of the following statements holds:
\begin{itemize}
\item[(1)] $\alpha,\beta\in\Fpbar \cap L$.
\item[(2)] $d=p^\ell$ for some positive integer $\ell$ and $\beta-\alpha\in\Fpbar\cap L$.
\item[(3)] $\alpha^d=\beta^d$.
\end{itemize}
Moreover, if either one of the conditions~(1)-(3) holds, then for each $\lambda\in\Lbar$, we have that $\alpha$ is preperiodic under the action of $f_\lambda$ if and only if $\beta$ is preperiodic under the action of $f_\lambda$. 
\end{theorem}

In the next Subsection, we discuss the significance of our result in the more general context of unlikely intersections.


\subsection{The principle of unlikely intersections}
\label{subsec:history}

At its heart, this principle asserts that inside an ambient algebraic variety $X$ (defined over some algebraically closed field $K$), a subvariety $V$ has a Zariski dense intersection  
\begin{itemize}
\item[(A)] either with an infinite family of subvarieties $Y_n$ of dimensions less than the codimension of $V$ in $X$, all sharing a certain geometric property $\mathcal{P}_1$;
\item[(B)] or with an infinite subset $\Gamma$ of $X(K)$ which satisfies a certain algebraic property $\mathcal{P}_2$,
\end{itemize}
only if $V$ itself has a special geometric property among the subvarieties of $X$. Some of the most outstanding questions in Diophantine geometry can be formulated in this key; we give below a couple of examples, each time working inside a semiabelian variety $X$ defined over an algebraically closed field $K$ of characteristic $0$:
\begin{itemize}
\item[(A)] when $Y_n$ are torsion translates of algebraic subgroups of $X$ of dimension less than the codimension of $V$, we recover the famous conjectures made by Bombieri-Masser-Zannier and Pink-Zilber; for a comprehensive discussion we refer the reader to the beautiul book of Zannier \cite{Umberto}. 
\item[(B)] when $\Gamma$ is a finite rank subgroup of $X(K)$, we recover the former classical conjectures of Manin-Mumford and of Mordell-Lang  (see \cite{Laurent, McQ, Vojta}).
\end{itemize}
Masser and Zannier \cite{M-Z-1, M-Z-2} opened a great new avenue of research by proving yet another striking result regarding unlikely intersections in an algebraic family of elliptic curves. Inspired by the approach from \cite{M-Z-1} (which also has a dynamical reformulation), Baker and DeMarco \cite{Matt} proved a first outstanding result for unlikely intersections in a purely dynamical context. So, given an integer $d\ge 2$ and given complex numbers $a$ and $b$, Baker-DeMarco \cite{Matt} prove that if there exist infinitely many $\lambda\in\C$ such that both $a$ and $b$ are preperiodic under the action of $f_\l(x)=x^d+\l$, then $a^d=b^d$. In other words, the infinite occurence of the  unlikely event that both $a$ and $b$ are preperiodic points for the same polynomial $f_\l$ can only happen if both $a$ and $b$ have the same iterates under the entire family of maps $\{f_\l\}_{\l\in\C}$; so, a very rigid global condition is derived from the existence of infinitely many discrete unlikely events. The result of \cite{Matt} had several extensions; we mention only few of the known results:
\begin{itemize}
\item for more general families of polynomials and starting points (see \cite{GHT-ANT, Matt-2, G-Y, F-G}), including families of polynomials parameterized by points in a higher dimensional space (see \cite{GHT-London, GHT-IMRN, GHN});
\item for certain families of rational maps (see \cite{Laura, GHT-London}), and also for arbitrary families of Latt\'es maps (see \cite{D-M}). 
\end{itemize}
Each time, the proof of any of the above results had two distinct parts:
\begin{itemize}
\item[(I)] first, one proves that a certain equidistribution theorem for points of small height holds for the given dynamical system, which leads to knowing that certain  canonical heights (suitably normalized) computed for the two starting points with respect to our family of maps are equal; 
\item[(II)] then using the equality of the above canonical heights, one derives the precise relation between the two starting points.
\end{itemize}

Now, the key ingredient for establishing part~(I) above comes from any of the equidistribution theorems of Baker-Rumely \cite{BR-2}, Chambert-Loir \cite{CL}, Favre-Rivera-Letelier \cite{FRL} or Yuan \cite{Yuan}. Verifying the hypotheses of the aforementioned equidistribution theorems is the difficult part and requires a detailed analysis of the arithmetical properties of the given dynamical system. Usually, completing step~(II) above is easier and it generally relies on two ingredients: a complex dynamics argument (which in turn uses crucially some key features of complex analytic functions, such as the Open Mapping Theorem), along with the  refined characterization provided by Medvedev and Scanlon \cite{Alice-Tom} of the subvarieties of $\A^N$, which are invariant under the coordinatewise action of $N$ one-variable polynomials.

All of the above results hold over fields of characteristic $0$, essentially because in positive characteristic one lacks completely the tools for dealing with the aforementioned step~(II). In the present paper, we obtain a first complete answer to an unlikely intersection problem for a \emph{dynamical system} in characteristic $p$.


\subsection{The picture in positive characteristic}
\label{subsec:his_p}

Overall, there are only a handful of results for the unlikely intersection principle in characteristic $p$. These known results are valid for Drinfeld modules (see \cite{B-Masser, B-Masser-2, GH-Acta, G-JNT}) since the Drinfeld modules are the natural vehicle in positive characteristic for many of the classical questions in arithmetic geometry, such as the Andr\'e-Oort conjecture (see \cite{Breuer}), the Bogomolov conjecture (see \cite{Bosser}),  the Mordell-Lang conjecture (see \cite{G3, GT-Compo}), the Manin-Mumford conjecture (see \cite{Scanlon-0}), the Siegel's theorem (see \cite{GT-Math.Ann}). Generally, if one tries to prove results in characteristic $p$ beyond the world of Drinfeld modules, then one encounters significant difficulties, especially in a purely dynamical setting.

In Theorem~\ref{thm:main} we establish the counterpart of the main result of \cite{Matt} in positive characteristic. The three different possibilities~(1)-(3) from Theorem~\ref{thm:main} show the distinct scenarios one has to deal with when working arithmetic questions in characteristic $p$; so, we have a trichotomy:
\begin{itemize}
\item[(1)] the case when the starting points $\alpha$ and $\beta$ live in $\Fpbar$ amounts to the so-called \emph{isotrivial} case that is always special in positive characteristic. 
\item[(2)] the case when $d=p^\ell$ is a power of the characteristic is also special since then each polynomial $f_\l(x)=x^d+\l$ from our family is an \emph{affine} map on $\mathbb{G}_a$, i.e., it is a  composition of an additive polynomial $x\mapsto x^{p^\ell}$ with a translate $x\mapsto x+\l$.
\item[(3)] in the \emph{generic} case, i.e., in the absence of cases~(1)-(2) above, then indeed the only possibility for $\alpha$ and $\beta$ to admit infinitely many parameters $\l$ such that both starting points are preperiodic under the action of $f_\l$ is when $\alpha^d=\beta^d$ (same as in characteristic $0$).
\end{itemize}

The fact that the isotrivial case in characteristic $p$ must be treated separately was a feature in many of the classical Diophantine problems, such as the Mordell-Lang conjecture (see \cite{Rahim-Tom}), or the Bogomolov problem (see \cite{G}). Furthermore, as discussed in \cite{Tom}, isotriviality plays an important role also for dynamical questions in characteristic $p$. So, it is natural that we have to consider separately the case~(1) in Theorem~\ref{thm:main}.
\begin{remark}
\label{rem:dynamical}
The case~(2) above appears due to the fact that when $d=p^\ell$, our family of polynomials $f_\l=x^{p^\ell}+\l$ commutes with additional polynomials (besides the identity map). In fact, given any translate $T_\xi(x)=x+\xi$ for some $\xi\in\Fpbar$, then $T_\xi$ commutes with $f^m_\l$, where $m$ is a positive integer so that $\xi\in \mathbb{F}_p^{\ell m}$. 
\end{remark}
Finally, we note that  a somewhat similar trichotomy as in the above cases~(1)-(3) appeared in \cite{GS-Cambridge} in the context of the Zariski Dense Orbit Conjecture studied once again over fields of positive characteristic.


\subsection{The strategy for our proof}
\label{subsec:explain}

We also prove (see Theorem~\ref{thm:main_2}) a generalization of Theorem~\ref{thm:main} by replacing the hypothesis that there exist infinitely many parameters $\lambda$ for which both $\alpha$ and $\beta$ are preperiodic under the action of $f_\l$ with the weaker hypothesis that (in a suitable product formula field $L$) there exist infinitely many parameters $\l_n$ such that 
\begin{equation}
\label{eq:1111}
\lim_{n\to\infty}\hhat_{f_{\l_n}}(\alpha)=\hhat_{f_{\l_n}}(\beta)=0;
\end{equation}
for more details regarding the global canonical heights $\hhat_{f_\l}$, see Subsection~\ref{subsec:global_height}. The fact that we can reduce in our Theorem~\ref{thm:main} to the case $L$ is a  product formula field is explained in Subsection~\ref{subsec:trdeg} (especially, see  Proposition~\ref{prop:trdeg}). Also, as noted in Remark~\ref{rem:clearly}, once $\alpha$ (or $\beta$) is preperiodic under the action of $f_\l$, then its global canonical height (with respect to $f_\l$) equals $0$; hence, the condition~\eqref{eq:1111} is weaker than the hypothesis from Theorem~\ref{thm:main}.

Similar to the proof of Baker-DeMarco \cite{Matt}, the first move is proving that the equidistribution theorem from \cite{BR} holds, which allows us to conclude that certain local canonical heights constructed with respect to the two starting points $\alpha$ and $\beta$ are equal (for more details, see Section~\ref{sec:heights} and also, see Theorem~\ref{thm:iff}). This part is essentially identical to the same analysis done in \cite{Matt}, and it refers to the so-called part~(I) (see our Subsection~\ref{subsec:history}) of any attempt to proving an unlikley intersection result in arithmetic dynamics. However, in order to get the precise relation between $\alpha$ and $\beta$ (the so-called part~(II), as explained in Subsection~\ref{subsec:history}), since the argument using complex analysis is no longer available in characteristic $p$, we need to employ a very detailed local analysis (at each nonarchimedean place) in order to derive the desired conclusion from Theorem~\ref{thm:main}.

\begin{remark}
\label{rem:Drinfeld}
The fact that the equidistribution theorem for points of small height can be applied to dynamical systems in positive characteristic was previously observed (see \cite{GH-Acta, G-JNT}). However, each time in the past it was prefered to work with a family of Drinfeld modules specifically to avoid the typical difficulties one encounters when dealing with Diophantine problems in positive characteristic (see cases~(1)-(2) in Theorem~\ref{thm:main}, also detailed in Subsection~\ref{subsec:his_p}). In particular, as conjectured in \cite{GH-Acta, G-JNT}, for families of Drinfeld modules one expects the same conclusion in terms of relations between the starting points, exactly as in the case of families of polynomials in characteristic $0$; in other words, the dynamics of a family of Drinfeld modules is almost oblivion of the fact that one works in characteristic $p$, as cases~(1)-(2) from Theorem~\ref{thm:main} do not appear in that setting.
\end{remark}

Now, going back to the plan of our paper, in order to state the equidistribution theorem that we will employ in our proof (see Theorem~\ref{thm:equi}), we need a technical setup both from the theory of Berkovich spaces and also from arithmetic dynamics; this is done in Section~\ref{sec:equi}. We continue by introducing canonical heights (both local and global) associated to our family of polynomials; this is done in Section~\ref{sec:heights}. Our results from Section~\ref{sec:heights} provide the technical background for obtaining the crucial Theorem~\ref{thm:iff} in Section~\ref{sec:equality}. Theorem~\ref{thm:iff} says that the existence of an infinite sequence of parameters $\l_n$ satisfying equation~\eqref{eq:1111} yields that for \emph{each} parameter $\l$ and for each nonarchimedean place $v$ of $L$, we have
\begin{equation}
\label{eq:equality_l}
\hhat_{v,\l}(\alpha)=\hhat_{v,\l}(\beta);
\end{equation}
for the precise definition of the local canonical heights $\hhat_{v,\l}$, we refer the reader to Section~\ref{sec:heights}. 

In Section~\ref{sec:proof} we prove Proposition~\ref{prop:main}, which says that assuming equation~\eqref{eq:equality_l} holds (for each place $v$ and each parameter $\l$), and also assuming that $d$ is a not a power of $p$ and that not both $\alpha$ and $\beta$ live in $\Fpbar$, then condition~(3) from Theorem~\ref{thm:main} must hold. Its proof requires a refined analysis of the valuations for $\alpha^d-\beta^d$, obtained by employing equation~\eqref{eq:equality_l} for suitably chosen parameters $\l$. A similar  strategy was previously used for obtaining the precise relation between starting points for other questions regarding the unlikely intersection principle in arithmetic dynamics (see \cite{GHT-IMRN, GHN} in characteristic $0$, and also, see \cite{GH-Acta, G-JNT} in characteristic $p$). However, each time, the technical arguments are different because one has to take into account the special features of the given dynamical system. As an aside, we note that our argument can be employed verbatim also for the original problem studied in \cite{Matt} (assuming, say, the starting points $a$ and $b$ live in a number field); however, the conclusion one obtains from the study of nonarchimedean absolute values is only that $a^d-b^d$ must be an algebraic integer.
 
Theorem~\ref{thm:iff} coupled with Proposition~\ref{prop:main}  prove the direct implication (which is the much harder part) from the conclusion of Theorem~\ref{thm:main}. Finally, in Section~\ref{sec:proofs}, we conclude our proof of Theorem~\ref{thm:main}. We actually state and prove the more general Theorem~\ref{thm:main_2} and show first how to deduce Theorem~\ref{thm:main} as a consequence of Theorem~\ref{thm:main_2}. The main part of Section~\ref{sec:proofs} is devoted to proving  Theorem~\ref{thm:main_2}; once again, the key ingredient is our Proposition~\ref{prop:main}.


\section{Equidistribution for points of small height}
\label{sec:equi}

As mentioned in Section~\ref{sec:intro}, we will need to apply the arithmetic
equidistribution discovered independently by Baker-Rumely \cite{BR-2},
Chambert-Loir \cite{CL} and Favre-Rivera-Letelier \cite{FRL}; when the base field is a
nonarchimedean field, the equidistribution theorem is best stated over
the Berkovich space associated to the underlying variety in question. 
We will introduce briefly the desired equidistribution theorem for points of small height (see Theorem~\ref{thm:equi}); for a comprehensive introduction to 
Berkovich spaces, we refer the reader to \cite{BR}. In our presentation, we use the approach of Baker-Rumely, which connects the equidistribution theorem to the
theory of arithmetic capacities. Hence, the material presented in this
Section~\ref{sec:equi} is mainly from the book~\cite{BR} by Baker and
Rumely.

So, following \cite[Definition~7.51]{BR}, we let $L$ be a field of characteristic $p$ endowed with a product
formula, i.e., there exists a set $\Omega_{L}$ of (pairwise inequivalent) absolute
values satisfying the following conditions:
\begin{itemize}
\item[(i)] for each nonzero $x\in L$, we have $|x|_v=1$ for all but finitely many $v\in\Omega_L$; and
\item[(ii)] for each nonzero $x\in L$, we have 
\begin{equation}
\label{eq:product_formula}
\prod_{v\in\Omega_L}|x|_v=1.
\end{equation}
\end{itemize}
We note that usually, one asks that the product formula~\eqref{eq:product_formula} holds in a slighly more general form: $\prod_{v\in\Omega_L}|x|_v^{N_v}=1$ for some given positive integers $N_v$; however, since all the absolute values from $\Omega_L$ are nonarchimedean, we can absorb the exponents $N_v$ in the definition of the respective absolute values $|\cdot |_v$ (see also \cite[Equation~2.2]{GH-Acta}). Furthermore, as mentioned in \cite[Chapter~7]{BR}, one does not require $L$ to be a global (function) field, but rather one needs that $L$ is a general product formula field (see equations~(i)-(ii) above). In particular, we can let $L_0$ be  the perfect closure of the rational function field (in one variable) over $\Fpbar$, i.e., 
\begin{equation}
\label{eq:L}
L_0:=\Fpbar\left(t, t^{1/p}, t^{1/p^2},\cdots, t^{1/p^n},\cdots\right)
\end{equation}
and then take $L$ be any finite extension of $L_0$; then $L$ is a product formula field. Indeed, each place of $\Fpbar(t)$ (which geomerically, corresponds   to a point of $\mathbb{P}^1(\Fpbar)$) extends uniquely to a place $w$ of $L_0$, thus making $L_0$ a product formula field. Above each given place $w$ of $L_0$ there exist finitely many places $v$ of $L$; we denote by $\Omega:=\Omega_L$ this set of places of $L$. Then $L$ is a product formula field with respect to $\Omega$. Furthermore, the separable closure $L^{\rm sep}$ of $L$ coincides with its algebraic closure $\Lbar$ (see also \cite[Remark~1.1]{GS-JNT}). Finally, we have the following fact: only the points in $\Fpbar$ are the points $x\in L$ which are integral at each place in $\Omega$, i.e., 
\begin{equation}
\label{eq:finite_field}
\text{if }|x|_v\le 1\text{ for each }v\in\Omega\text{, then }x\in\Fpbar.
\end{equation}

In the rest of this Section, we work with an arbitrary product formula field $L$; however, the relevant case for our results is a finite extension of the field from~\eqref{eq:L}. 
Now, for each $v\in\Omega_L$, we let $\C_v$ be an algebraically closed field containing $L$, which is also complete with respect to a fixed extension of $|\cdot |_v$ to $\C_v$.    
Let $\Aberk{\C_v}$ denote the Berkovich affine  line over
$\C_v$ (see \cite{BR} or  \cite[Section~2]{Matt} for more 
details).  In order to apply the main equidistribution result from \cite[Theorem
7.52]{BR}, we recall briefly the potential theory on the 
affine line over $\C_v$.   The right setting for nonarchimedean potential theory is the
potential theory on $\Aberk{\C_v}$ developed in \cite{BR}. We
quote here part of a nice summary of the theory 
from \cite[Section~2]{Matt} without going into details (we
refer the reader to \cite{BR, Matt}  for all the
details and proofs). 

So, let $E$ be a compact subset of
$\Aberk{\C_v}$. Then analogous to the complex case, the logarithmic
capacity $\g(E) = e^{-V(E)}$ and the Green's function $G_E$ of $E$
relative to $\infty$ can be defined where $V(E)$ is the infimum of the 
\emph{energy integral} with respect to all possible probability
measures supported on $E$. More precisely,
$$V(E)=\inf_{\mu}\int\int_{E\times E} -\log\delta(x,y) d\mu(x)d\mu(y),$$
where the infimum is computed with respect to all probability measures $\mu$ supported on $E$, while for $x, y\in \Aberk{\C_v},$  the function $\delta(x,y)$ 
is the \emph{Hsia kernel} (see \cite[Proposition~4.1]{BR}).
If $\g(E) > 0$, then the exists a unique probability measure $\mu_E$, also called the \emph{equilibrum measure on $E$}, 
attaining the infimum of the energy integral. Furthermore, the support of
$\mu_E$ is contained in the boundary of the unbounded component of
$\Aberk{\C_v}\setminus E$.  The Green's function 
$G_E(z)$ of $E$ relative to infinity is a well-defined nonnegative
real-valued subharmonic function on $\Aberk{\C_v}$ which is harmonic on $\Aberk{\C_v}\setminus E$ (in the sense of
\cite[Chapter~8]{BR}). 

The following result (see \cite[Lemma 2.5]{Matt}) summarizes the key features of the Green's function.
\begin{lemma}
  \label{green function}
Let $E$ be a compact subset of $\Aberk{\C_v}$ and let $U$ be the unbounded component of
$\Aberk{\C_v}\setminus E$. 

\begin{itemize}
\item[(1)]
If $\g(E)>0$ (i.e. $V(E)<\infty$), then
  $G_E(z) = V(E) + \log |z|_v $ for all $z\in \Aberk{\C_v}$ such that
  $|z|_v$ is sufficiently large.

\item[(2)]
If $G_E(z) = 0$ for all $z\in E,$ then $G_E$ is continuous on
  $\Aberk{\C_v},$ $\supp(\mu_E) = \partial U$ and $G_E(z) > 0$ if and
  only if $z\in U.$  

\item[(3)]
  If $G :\Aberk{\C_v}\to \R$ is a continuous subharmonic function which
  is harmonic on $U,$ identically zero on $E$, and such that $G(z) -
  \log^+|z|_v$ is bounded, then $G = G_E$. Furthermore, if $G(z)=\log|z|_v + V + o(1)$ (as $|z|_v\to\infty$) for some $V<\infty$, then $V(E)=V$ and so, $\g(E)=e^{-V}$.
\end{itemize}
\end{lemma}

To state the equidistribution result from \cite{BR}, we
consider the compact \emph{Berkovich ad\`elic sets} which are of the
following form
\begin{equation}
\label{eq:E}
\E := \prod_{v\in \Omega_L}\, E_v,
\end{equation}
where $E_v$ is a non-empty compact subset of $\Aberk{\C_v}$ for each
$v\in \Omega_L$ and furthermore, $E_v$ is the closed unit disk $\cD(0,1)$ in
$\Aberk{\C_v}$ for all but finitely many $v\in \Omega_L$. The
\emph{logarithmic capacity} $\g(\E)$ of $\E$ is defined as follows
\begin{equation}
\label{logarithmic capacity formula 0}
\g(\E) = \prod_{v\in \Omega_L}\,\g(E_v). 
\end{equation}
Note that in \eqref{logarithmic capacity formula 0} there is a finite product as for all but finitely many 
$v\in\Omega_L$, we have $\g(E_v) = \g(\cD(0,1)) = 1$.   Let $G_v := G_{E_v}$  be
the Green's function of $E_v$ 
relative to $\infty$ for each $v\in \Omega_L$. For every $v\in\Omega_L$, we
fix an embedding of the separable closure $L^{\rm sep}$ of $L$ into $\C_v$. Let 
$S\subset L^{\rm sep}$ be any finite subset that is invariant under the action
of the Galois group 
$\gal(L^{\rm sep}/L)$. We define the height $h_{\E}(S)$ of $S$ relative to $\E$
by
\begin{equation}
\label{def ade hig}
h_{\E}(S) = \sum_{v\in\Omega_L} \left(\frac{1}{|S|}\sum_{z\in S}G_v(z)\right).
\end{equation}
Note that this definition is independent of the particular embedding
$L^{\rm sep}$ into $\C_v$ that we choose at each place $v\in \Omega_L$. Finally, for each $v\in\Omega_L$, we let $\mu_v$ be the equilibrum measure on $E_v$. 
The following is a special case of the equidistribution result 
\cite[Theorem~7.52]{BR} that we need for our application.
\begin{theorem}
 \label{thm:equi}
With the above notation, let $\E = \prod_{v\in \Omega} E_v$ be a compact Berkovich ad\`elic set
 with $\g(\E)=1.$ Suppose that 
 $S_n$ is a sequence of $\gal(L^{\rm sep}/L)$-invariant finite subsets of
 $L^{\rm sep}$ with $|S_n|\to \infty$ and $h_{\E}(S_n) \to 0$ as $n\to
 \infty$.  
For each $v\in\Omega_L$ and for each $n$ let $\d_n$ be the discrete
 probability measure supported equally on the elements of $S_n$. Then the
 sequence of measures $\{\d_n\}$ converges weakly to $\mu_v$ the
 equilibrium measure on $E_v$. 
\end{theorem}


\section{Dynamics and heights associated to our family of polynomials}
\label{sec:heights}

Throughout this Section, we let $L_0$ be the perfect closure of $\Fpbar(t)$ (see its definition from~\eqref{eq:L}) and then  we let $L$ be a given finite extension of $L_0$. Then each finite extension of $L$ is separable, i.e., $L^{\rm sep}=\Lbar$; so, from now on, we fix an algebraic closure $\Lbar$ of $L$. Also, for the sake of simplifying our notation, we let $\Omega:=\Omega_L$ be the set of inequivalent places  of $L$ witnessing the fact that $L$ is a product formula field.


\subsection{Preperiodic parameters for a given starting point}
\label{subsec:prep}

We let $d\ge 2$ be an integer. 
We work
with a family of polynomials  
as given  in Theorem~\ref{thm:main}, i.e. $f_\lambda(x)=x^d+\lambda$ parameterized by $\lambda\in\Lbar$. Given $\gamma\in L$, we define 
\begin{equation}
\label{eq:P}
P_{n,\gamma}(\lambda):=f^n_\lambda(\gamma)\text{ for each 
$n\in\N$;} 
\end{equation}
then $P_{n,\gamma}(\lambda)$ is a polynomial in $\lambda$.   A simple induction on $n$ yields the following result.

\begin{lemma}
\label{lem:degree in l}
With the above hypothesis, for each $n\in\N$, the polynomial $P_{n,\gamma}(\l)$ is monic and has degree
$d^{n-1}$ in $\l$. 
\end{lemma}

\begin{remark}
\label{rem:prep}
We immediately obtain as a corollary of Lemma~\ref{lem:degree in l} the fact that $\gamma$ is not preperiodic for the entire family of polynomials $f_\lambda$.  Furthermore we obtain that if $\gamma$ is preperiodic for $f_\lambda$, then $\l\in\Lbar$.
\end{remark}


\subsection{Generalized Mandelbrot sets}

From now on, in Section~\ref{sec:heights}, we fix a place $v\in \Omega$. 
 
Following the same approach as in  \cite{Matt}, one defines the 
{\em generalized Mandelbrot set} $M_{\gamma,v}\subset \Aberk{\C_v}$ associated to $\gamma$; roughly speaking, $M_{\gamma,v}$ is the subset of $\C_v$ consisting of all $\l\in
\C_v$ such that $P_{n,\gamma}(\l)$ is $v$-adic bounded, as we let $n\to\infty$.   

Let $\l\in \C_v$ and define the local canonical height
$\hhat_{v,\lambda}(x)$ of $x\in \C_v$ with respect to the polynomial $f_\lambda$; more precisely, we have the formula 
\begin{equation}
\label{eq:def}
\hhat_{v,\lambda}(x) :=\lim_{n\to\infty}
\frac{\log^+|f_\lambda^n(x)|_v}{d^{n}},
\end{equation}
where $\log^+(z)=\log\max\{z,1\}$ for each real number $z$. Clearly, $\hhat_{v,\lambda}(x)$ is a continuous function of both $\l$
and $x$ on $\C_v$.  Also, we will be using the following easy fact:
\begin{equation}
\label{eq:def_2}
\hhat_{v,\lambda}(x)=\frac{\hhat_{v,\lambda}(f_\lambda^m(x))}{d^m}\text{ for each }m\in\N\text{ and for each }x\in\C_v.
\end{equation}

As $\C_v$ is a dense subspace of $\Aberk{\C_v}$,
continuity in $\l$ implies that  the canonical local height function 
$\hhat_{v,\lambda}(\gamma)$ has a natural extension on $\Aberk{\C_v}$ (note that the topology on
$\C_v$ is the restriction of the weak topology on $\Aberk{\C_v}$, so any
continuous function
on $\C_v$ will automatically have a unique extension to $\Aberk{\C_v}$). Then 
$\l\in M_{\gamma,v}$ if and only if $\hhat_{v,\lambda}(\gamma) = 0$.  Thus, 
$M_{\gamma,v}$ is a closed subset of $\Aberk{\C_v}$; in fact, the following is
true (as previously proved in \cite{Matt}). 

\begin{proposition}
\label{prop:bounded mandelbrot} 
$M_{\gamma,v}$ is a compact subset of $\Aberk{\C_v}$.  
\end{proposition}

\begin{proof}
Since we already know that $M_{\gamma,v}$ is a closed subset of the locally
compact space $\Aberk{\C_v}$, then in order to prove
Proposition~\ref{prop:bounded mandelbrot} it suffices to show that
$M_{\gamma,v}$ is a bounded subset of $\Aberk{\C_v}$. This last fact follows immediately from Lemma~\ref{lem:easy},~part~(iii).
\end{proof}

The following Lemma~\ref{lem:easy} is not only used in the proof of Proposition~\ref{prop:bounded mandelbrot}, but it is also repeatedly used throughout Section~\ref{sec:proof}; its proof is easy but its findings are important. 
\begin{lemma}
\label{lem:easy}
Let $\gamma,\lambda\in \C_v$.  
\begin{itemize}
\item[(i)] If $\max\{|\lambda|_v,|\gamma|_v\}\le 1$, then 
$$\hhat_{v,\lambda}(\gamma)=0.$$
\item[(ii)] If $|\gamma|_v^d>\max\{1,|\lambda|_v\}$, then 
\begin{equation}
\label{eq:def_0}
\hhat_{v,\lambda}(\gamma)= \log|\gamma|_v>0.
\end{equation}
\item[(iii)] If $|\lambda|_v>\max\left\{1,|\gamma|_v^{d}\right\}$, then $$\hhat_{v,\lambda}(\gamma)=\frac{\log|\lambda|_v}{d}>0.$$
\end{itemize}
\end{lemma}

\begin{proof}[Proof of Lemma~\ref{lem:easy}.]
We first note that conclusion~(i) is immediate since knowing that both $\lambda$ and $\gamma$ are integral at the place $v$ yields that each $f_\lambda^n(\gamma)$ is integral at $v$, thus showing that $\hhat_{v,\lambda}(\gamma)=0$.

Next, we work under the hypotheses from part~(ii). The fact that $|\gamma|_v^d >\max\{1, |\lambda|_v\}$ yields that
$$|f_\lambda(\gamma)|_v=|\gamma^d+\lambda|_v=|\gamma|_v^d> |\gamma|_v.$$
An easy induction on $n$ shows that for each $n\ge 1$, we have that
$$\left|f_\lambda^n(\gamma)\right|_v = |\gamma|_v^{d^{n}};$$
then the desired conclusion in part~(ii) follows.

Finally, part~(iii) is a consequence of part~(ii) because the inequality 
$|\lambda|>\max\left\{1,|\gamma|_v^{d}\right\}$ 
yields
\begin{equation}
\label{eq:l-g}
\left|f_\lambda(\gamma)\right|_v=\left|\gamma^d+\lambda\right|_v=|\lambda|_v >|\lambda|_v^{\frac{1}{d}}.
\end{equation}
Equation~\eqref{eq:l-g} allows us to apply the conclusion from part~(ii) to the point $f_\l(\gamma)$ and the parameter $\l$ and thus, we get $$\hhat_{f_\l}(f_\l(\gamma))=|f_\l(\gamma)|_v=|\l|_v.$$  Then equation~\eqref{eq:def_2} yields the desired conclusion in Lemma~\ref{lem:easy},~part~(iii).  
\end{proof}


\subsection{The logarithmic capacities of the generalized Mandelbrot sets}

Next our goal is to compute
 the logarithmic capacities of the $v$-adic generalized 
 Mandelbrot sets $M_{\gamma,v}$ associated to $\gamma$ for our given family f polynomials $f_\lambda$.

\begin{theorem}
  \label{mandelbrot capacity}
The logarithmic capacity of
$M_{\gamma,v}$ is $\g(M_{\gamma,v}) =1 $. 
\end{theorem}

The strategy for the proof of Theorem~\ref{mandelbrot capacity} is to
construct a continuous subharmonic function $G_{\l,v} : \Aberk{\C_v}\to \R$
satisfying Lemma~\ref{green function}~(3); the technical steps follow identically as in the proof of the similar result from \cite{Matt}. So, we let 
\begin{equation}
\label{eq:G}
G_{\gamma,v}(\lambda):= \lim_{n\to\infty}\frac{\log^+|f^n_\lambda(\gamma)|_v}{d^{n-1}}=d\cdot  \hhat_{v,\lambda}(\gamma).
\end{equation}
Note that $G_{\gamma,v}(\l) \ge 0$ for all $\l\in \Aberk{\C_v}$; also,  $\l\in M_{\gamma,v}$ if and only if $G_{\gamma,v}(\l) = 0$.   

\begin{lemma}
\label{green's function for mandelbrot}
$G_{\gamma,v}$ is the Green's function for $M_{\gamma,v}$ relative to $\infty.$ 
\end{lemma}

\begin{proof}
The proof is essentially the same as the proof of
\cite[Proposition~3.7]{Matt} since we deal with the same family of polynomials.   We also note that generalizations of this result hold for other families of polynomials (see \cite{GHT-ANT, Matt-2}), including over fields of positive characteristic (see \cite[Lemma~4.7]{GH-Acta}). 
\end{proof}

Now we are ready to prove Theorem~\ref{mandelbrot capacity}.
\begin{proof}[Proof of  Theorem~\ref{mandelbrot capacity}.]
Lemma~\ref{lem:easy}~(iii) yields that 
\begin{equation}
\label{eq:G 2}
 G_{\gamma,v}(\l)  = \log|\l|_v\text{, for $|\l|_v$ sufficiently large.}  
\end{equation}
Combining Lemma~\ref{green function}~(3) with equation~\eqref{eq:G 2}, we find that $V(M_{\gamma,v}) = 0$. Hence,
  the logarithmic capacity of $M_{\gamma,v}$ is $1$, as desired. 
\end{proof}


\subsection{The generalized ad\`elic Mandelbrot set}

Let us call $\M_{\gamma} = \prod_{v\in
\Omega} M_{\gamma,v}$  the \emph{generalized ad\`elic Mandelbrot set} associated to
$\gamma$. As a corollary to Theorem~\ref{mandelbrot capacity} we see that
$\M_{\gamma}$ satisfies the hypothesis of Theorem~\ref{thm:equi}.
\begin{corollary}
 \label{adelic mandelbrot}
For all but finitely many nonarchimedean places $v$, we have that $M_{\gamma,v}$ is the closed unit disk $\cD(0;1)$ in $\Aberk{\C_v}$; furthermore $\g(\M_{\gamma}) = 1$.
\end{corollary}

\begin{proof}
Indeed, using Lemma~\ref{lem:easy},~parts~(i)~and~(iii), we get that for each place $v\in\Omega$ where $\gamma$ is integral (i.e., $|\gamma|_v\le 1$), we have that  $M_{\gamma,v}=\cD(0;1)$. The ``furthermore'' assertion follows immediately from Theorem~\ref{mandelbrot capacity}. 
\end{proof} 


\subsection{Global canonical heights}
\label{subsec:global_height}

For each $\lambda\in\Lbar$ (again note that $\Lbar=L^{\rm sep}$), we will use the notation 
\begin{equation}
\label{eq:height_adelic}
h_{\M_{\gamma}}(\l):=h_{\M_{\gamma}}(S)\text{ where $S$ is the $\gal(L^{\rm sep}/L)$-orbit of $\l$}.
\end{equation}  
The notation from \eqref{eq:height_adelic} is connected to the global canonical height associated to the polynomials $f_\lambda$. 

\begin{definition}
\label{def:global_height}
For each $x\in\Lbar$, we define its Weil height as
\begin{equation}
\label{eq:Weil_height}
h(x):=\frac{1}{[L(x):L]}\cdot \sum_{v\in\Omega}\sum_{y\in \gal(L^{\rm sep}/L)\cdot x} \log^+|y|_v.
\end{equation}
For each $\lambda\in\Lbar$, we define the global canonical height of $x\in\Lbar$ with respect to the polynomial $f_\lambda$ as
\begin{equation}
\label{eq:global_height}
\hhat_{f_\lambda}(x)=\lim_{n\to\infty}\frac{h\left(f^n_\lambda(x)\right)}{d^n}.
\end{equation}
\end{definition}
%
%
\begin{remark}
\label{rem:clearly}
If $\gamma$ is preperiodic under the action of $f_\l$, then it is immediate to see (based on equation~\eqref{eq:global_height}) that $\hhat_{f_\l}(\gamma)=0$ (since there are finitely many distinct points $f^n_\l(\gamma)$). 

However, using \cite[Theorem~B]{Rob}, one can also establish the converse statement as well, i.e., once $\hhat_{f_\l}(\gamma)=0$, then $\gamma$ must be preperiodic under the action of $f_\l$. Indeed, as long as $\l\notin\Fpbar$, then $f_\l$ is not isotrivial and therefore, \cite[Theorem~B]{Rob} shows that a point is preperiodic if and only if its canonical height equals $0$. Finally, if $\l\in\Fpbar$, then it is immediate to see that $\gamma$ is preperiodic if and only if also $\gamma\in\Fpbar$. Similarly, if $\hhat_{f_\l}(\gamma)=0$ (and $\l\in\Fpbar$), then we must have that $|\gamma|_v\le 1$ for each place $v\in\Omega$ (see Lemma~\ref{lem:easy}~(ii)) and therefore, we must also have that $\gamma\in\Fpbar$ (see \eqref{eq:finite_field}).  
\end{remark}

The following fact follows from the decomposition of the global canonical height as a sum of local canonical heights.
\begin{lemma}
\label{lem:global_height}
Let $\gamma\in L$. Then for each $\l\in\Lbar$, we have $h_{\M_\gamma}(\l)= d\cdot \hhat_{f_\l}(\gamma)$.
\end{lemma}

\begin{proof}
Combining equations~\eqref{eq:height_adelic}~and~\eqref{def ade hig}, we have
\begin{equation}
\label{eq:111}
h_{\M_\gamma}(\l)=\frac{1}{[L(\l):L]}\cdot \sum_{v\in\Omega}\sum_{\sigma\in\gal(L(\l)/L)} G_{\gamma,v}(\sigma(\lambda)).
\end{equation}
Then equation~\eqref{eq:111} combined with equation~\eqref{eq:G} yields
\begin{equation}
\label{eq:112}
h_{\M_\gamma}(\l)= \frac{1}{[L(\l):L]}\cdot \sum_{v\in\Omega}\sum_{\sigma\in\gal(L(\l)/L)} d\cdot \hhat_{v,\sigma(\l)}(\gamma)
\end{equation}
and then re-arranging the sum from \eqref{eq:112} since only finitely many terms are nonzero leads to
\begin{equation}
\label{eq:113}
h_{\M_\gamma}(\l)= \frac{d}{[L(\l):L]}\cdot \sum_{\sigma\in\gal(L(\l)/L)}\sum_{v\in\Omega} \lim_{n\to\infty}\frac{\log^+\left|f^n_{\sigma(\l)}(\gamma)\right|_v}{d^n}.
\end{equation}
Using that for each $\l$ and $\sigma$, there are only finitely many places $v\in\Omega$ such that $\sigma(\l)$ and $\gamma$ are not $v$-adic integral, then equation~\eqref{eq:113} can be re-written as follows (see also Lemma~\ref{lem:easy}~(i))  
\begin{equation}
\label{eq:114}
h_{\M_\gamma}(\l)= \frac{d}{[L(\l):L]}\cdot \lim_{n\to\infty} \frac{1}{d^n}\cdot \sum_{\sigma\in\gal(L(\l)/L)}  \sum_{v\in\Omega} \log^+\left|f^n_{\sigma(\l)}(\gamma)\right|_v.
\end{equation}
Now, using the fact that $\sigma(f_\l^n(\gamma))=f^n_{\sigma(\l)}(\gamma)$ for each $\sigma\in\gal(L^{\rm sep}/L)$ (since $\gamma\in L$), coupled with the definition of the Weil height for  $f^n_{\l}(\gamma)$ (see equation~\eqref{eq:Weil_height}), we obtain 
$$h_{\M_\gamma}(\l)=d\cdot \lim_{n\to\infty} \frac{h\left(f^n_\lambda(\gamma)\right)}{d^n} = d\cdot \hhat_{f_\lambda}(\gamma).$$
This concludes our proof for Lemma~\ref{lem:global_height}.
\end{proof}


\section{Equality of the respective local canonical heights}
\label{sec:equality}

We continue with the notation as in Section~\ref{sec:heights}. In particular, $L$ is a finite extension of the perfect closure of the rational function field in one variable over $\Fpbar$ (see \eqref{eq:L}). Also, for any point $\gamma\in L$, we construct the generalized ad\`elic Mandelbrot set $\M_\gamma$ and then define the associated height $h_{\M_\gamma}$.  

The following result is the key technical ingredient which we extract from Theorem~\ref{thm:equi}.

\begin{theorem}
\label{thm:iff}
Let $L$, $f_\l$, $\hhat_{f_\l}$, $\hhat_{v,\l}$ be defined as in Section~\ref{sec:heights}; also, let $\alpha,\beta\in L$. 
Assume there exists an infinite sequence  $\{\lambda_n\}$ in $\Lbar$ with the property that 
\begin{equation}
\label{eq:Bogomolov}
\lim_{n\to\infty} \hhat_{f_{\l_n}}(\alpha)=\lim_{n\to\infty} \hhat_{f_{\l_n}}(\beta)=0.
\end{equation}
Then for each $\lambda\in\Lbar$ and for each $v\in\Omega$, we have that $\hhat_{v,\lambda}(\alpha)=\hhat_{v,\lambda}(\beta)$.
\end{theorem}

\begin{proof}
Lemma~\ref{lem:global_height} yields that for each $\l$, we have
$$h_{\M_\alpha}(\l)=d\cdot \hhat_{f_\l}(\alpha)\text{ and }h_{\M_\beta}(\l)=d\cdot  \hhat_{f_\l}(\beta),$$
where $h_{\M_\alpha}$ and $h_{\M_\beta}$ are constructed as in Section~\ref{sec:heights} with respect to the generalized ad\`elic Mandelbrot sets $\M_\alpha$ and $\M_\beta$. Hence, equation~\eqref{eq:Bogomolov} yields 
\begin{equation}
\label{eq:4000}
\lim_{n\to\infty}h_{\M_\alpha}(\l_n)=\lim_{n\to\infty}h_{\M_\beta}(\l_n)=0.
\end{equation}
Equation~\eqref{eq:4000} shows that the hypothesis from Theorem~\ref{thm:equi} holds and therefore, we conclude that $M_{\alpha,v}=M_{\beta,v}$ for each place $v\in \Omega$. 
Indeed,  for each $n\in\N$, we may
take $S_n$ be the union of the sets of Galois conjugates for $\l_m$
for all $1\le m\le n$. Clearly $|S_n|\to\infty$ as $n\to\infty$, and
also each $S_n$ is $\gal(L^{\rm sep}/L)$-invariant. Thus, we obtain
that $\mu_{M_{\alpha,v}}=\mu_{M_{\beta,v}}$  for each $v\in \Omega$ and since they are both
supported on $M_{\alpha,v}$ (resp. $M_{\beta,v}$), we also get that
$M_{\alpha,v}=M_{\beta,v}$. It follows that the two Green's functions 
$G_{\alpha,v}$ and $G_{\beta,v}$ for $M_{\alpha,v}$ and $M_{\beta,v}$ are the same. 
By the definitions of $G_{\alpha,v}$ ($G_{\beta,v}$ respectively)
(see equation~\eqref{eq:G}), for each $v\in \Omega$  we have 
$$
d\cdot \hhat_{v,\l}(\alpha) =    G_{\alpha,v}(\l)
= G_{\beta,v}(\l) = d\cdot \hhat_{v,\l}(\beta) \quad \text{for {\em all} $\l \in \C_v$.}
$$
This concludes our proof of Theorem~\ref{thm:iff}.
\end{proof}


\section{Proof of the precise relation between the starting points}
\label{sec:proof}

In this Section we prove the following result.

\begin{proposition}
\label{prop:main}
Let $L_0:=\Fpbar\left(t,t^{1/p},t^{1/p^2},\cdots, t^{1/p^n},\cdots \right)$ and let $L$ be a finite extension of $L_0$. We denote by $\Omega:=\Omega_L$ be the set of inequivalent places of $L$. We let $\alpha,\beta\in L$, not both of them contained in $\Fpbar$. Let $d\ge 2$ be an integer, which is not a power of the prime $p$. We let 
$$f_\lambda(x):=x^d+\lambda$$
be a family of polynomials parameterized by $\lambda\in \Lbar$. As in Section~\ref{sec:heights}, for each $\lambda\in\Lbar$ and for each place $v\in\Omega$, we let
$$\hhat_{v,\lambda}(\alpha)=\lim_{n\to\infty} \frac{\log^+\left|f_\lambda^n(\alpha)\right|_v}{d^n}\text{ and } \hhat_{v,\lambda}(\beta)=\lim_{n\to\infty} \frac{\log^+\left|f_\lambda^n(\beta)\right|_v}{d^n}.$$
If for each $\lambda\in \Lbar$ and for each place $v\in\Omega$, we have that 
\begin{equation}
\label{eq:hypothesis}
\hhat_{v,\lambda}(\alpha)=\hhat_{v,\lambda}(\beta), 
\end{equation}
then we must have that  $\alpha^d=\beta^d$. 
\end{proposition}

Proposition~\ref{prop:main} constitutes the bridge in our arguments between Theorem~\ref{thm:iff} and Theorem~\ref{thm:main} (see also Theorem~\ref{thm:main_2}).


\subsection{The strategy for proving Proposition~\ref{prop:main}}  
From now on, we work under the hypotheses from Proposition~\ref{prop:main}. We split its proof into   Subsections~\ref{subsec:0},~\ref{subsec:first},~\ref{subsec:second}~and~\ref{subsec:third}.

So, we let $S$ be the (finite) set of places $v\in\Omega$ with the property that 
\begin{equation}
\label{eq:2112}
\max\{|\alpha|_v,|\beta|_v\}>1. 
\end{equation}
Note that our hypothesis from Proposition~\ref{prop:main} that not both $\alpha$ and $\beta$ live in $\Fpbar$  yields that $S$ is a \emph{nonempty} set. Our strategy will be to prove that 
\begin{equation}
\label{eq:first_step}
\left|\alpha^d-\beta^d\right|_v\le 1\text{ for each }v\in S.
\end{equation}
Indeed, since $S$ consists of all the places $v$ where $\alpha$ or $\beta$ is  not $v$-adic integral (see inequality~\eqref{eq:2112}), then the only places of $\Omega$ for which $\alpha^d-\beta^d$ may not be a $v$-adic integer are exactly the ones from the set $S$. So, inequality~\eqref{eq:first_step} would prove that $\alpha^d-\beta^d$ is integral at each place $v\in\Omega$. Due to the product formula~\eqref{eq:product_formula} on $L$ (see also \eqref{eq:finite_field}), this means that $\alpha^d-\beta^d\in\Fpbar$, which is sufficient to deduce that $\alpha^d=\beta^d$ if $d=2$ (see Lemma~\ref{lem:d=2}). Now, in the case $d>2$, we can actually prove that the inequality in~\eqref{eq:first_step} is strict; this is sufficient to deduce that $\alpha^d=\beta^d$ (see Lemma~\ref{lem:d>2}). We also note (see Remark~\ref{rem:not_p_power}) that it is exactly in the last part of our argument (the proof of Lemma~\ref{lem:5}) where we employ the hypothesis from Proposition~\ref{prop:main} that $d$ is not a power of $p$.

In order to deduce the inequality~\eqref{eq:first_step}, we employ the hypothesis~\eqref{eq:hypothesis} from Proposition~\ref{prop:main} for various well-chosen $\lambda$'s in $\Lbar$. Also, we first prove that
\begin{equation}
\label{eq:second_step}
\left|\alpha^d-\beta^d\right|_v\le |\alpha|_v=|\beta|_v\text{ for each }v\in S;
\end{equation}
this is done in Subsection~\ref{subsec:0}.


\subsection{First step in the proof of Proposition~\ref{prop:main}}
\label{subsec:0}

In this Subsection, we will establish~\eqref{eq:second_step}. We first prove the following easy fact which will be used repeatedly in our proof of Proposition~\ref{prop:main}.
\begin{lemma}
\label{lem:2}
For each place $v\in S$, we have $|\alpha|_v=|\beta|_v>1$.
\end{lemma}

\begin{proof}[Proof of Lemma~\ref{lem:2}.]
The desired conclusion is an easy corollary of Lemma~\ref{lem:easy},~parts~(i)-(ii) using $\lambda=0$ and $v\in S$ (see~\eqref{eq:2112}), along with the hypothesis~\eqref{eq:hypothesis} of Proposition~\ref{prop:main}. 
\end{proof}

\begin{corollary}
\label{cor:easy}
With the hypothesis as in Proposition~\ref{prop:main}, we have that neither $\alpha$ nor $\beta$ live in $\Fpbar$.
\end{corollary}

\begin{proof}
Indeed, Lemma~\ref{lem:easy} shows that \emph{both} $\alpha$ and $\beta$ are not integral at the places from $S$; hence neither one can live in $\Fpbar$. 
\end{proof}
In particular, Corollary~\ref{cor:easy} yields that $\alpha$ and $\beta$ are nonzero. 

Next Lemma will finish the proof for the assertion from~\eqref{eq:second_step}.

\begin{lemma}
\label{lem:3}
For each place $v$ in $S$, we have that $|\beta^d-\alpha^d|_v\le |\alpha|_v$.
\end{lemma}

\begin{proof}[Proof of Lemma~\ref{lem:3}.]
We argue by contradiction and so, we assume that $|\beta^d-\alpha^d|_v> |\alpha|_v$ and we will derive a contradiction. 

Indeed, we consider $\lambda_0:=-\alpha^d$ and then apply Lemma~\ref{lem:easy}  for $\lambda_0$ and $\gamma:=f_{\lambda_0}(\beta)=\beta^d-\alpha^d$; since $|\beta^d-\alpha^d|_v^d>|\lambda_0|_v=|\alpha|_v^d$ (according to our assumption), Lemma~\ref{lem:easy}~(ii) yields that
$$\hhat_{v,\lambda_0}(\gamma)=\log\left|\beta^d-\alpha^d\right|_v.$$
But then using the fact that $\hhat_{v,\lambda_0}(\beta)=\frac{\hhat_{v,\lambda_0}(\gamma)}{d}$ (see equation~\eqref{eq:def_2}), we obtain that
\begin{equation}
\label{eq:1}
\hhat_{v,\lambda_0}(\beta)=\frac{\log|\beta^d-\alpha^d|_v}{d}.
\end{equation}
On the other hand, we compute
$$f_{\lambda_0}(\alpha)=0\text{ and }f^2_{\lambda_0}(\alpha)=-\alpha^d.$$
Then using again Lemma~\ref{lem:easy}~(ii), this time for $\lambda_0$ and $-\alpha^d$ (note that $|\alpha^d|^d_v>|\alpha^d|_v$), we conclude that 
$$\hhat_{v,\lambda_0}(-\alpha^d)=\log|\alpha^d|_v.$$
But then again using equation~\eqref{eq:def_2}, we get that
\begin{equation}
\label{eq:2}
\hhat_{v,\lambda_0}(\alpha)=\frac{\hhat_{v,\lambda_0}(f^2_{\lambda_0}(\alpha))}{d^2} = \frac{\log|\alpha|_v}{d}.
\end{equation}
However, our assumption that $|\alpha^d-\beta^d|_v>|\alpha|_v$ coupled with equations~\eqref{eq:1}~and~\eqref{eq:2} contradict the main hypothesis from our Proposition~\ref{prop:main} that $\hhat_{v,\lambda_0}(\alpha)=\hhat_{v,\lambda_0}(\beta)$. In conclusion, indeed, we must have that 
\begin{equation}
\label{eq:3}
|\alpha^d-\beta^d|_v\le |\alpha|_v=|\beta|_v,
\end{equation}
for each place $v$ in $S$. 
\end{proof}


\subsection{Second step in the proof of Proposition~\ref{prop:main}}
\label{subsec:first}

The inequality from equation~\eqref{eq:second_step} says that $|\alpha^d-\beta^d|_v$ is \emph{much smaller} than one would expect it to be; i.e., since $|\alpha|_v=|\beta|_v>1$ (for $v\in S$), then \emph{normally} one would expect $|\alpha^d-\beta|_v$ to be \emph{larger} than $|\alpha|_v=|\beta|_v$. The next Lemma refines further the inequality from~\eqref{eq:second_step} showing that we actually have a \emph{strict}  inequality in \eqref{eq:second_step}. 
\begin{lemma}
\label{lem:4}
For each place $v$ in $S$, we have that $|\beta^d-\alpha^d|_v<|\alpha|_v$.
\end{lemma}

\begin{proof}[Proof of Lemma~\ref{lem:4}.]
We argue by contradiction and therefore, assume that $|\alpha^d-\beta^d|_v\ge |\alpha|_v$. Then Lemma~\ref{lem:3} yields that actually $|\alpha^d-\beta^d|_v=|\alpha|_v$.
\begin{claim}
\label{claim:key}
There exists some nonzero $\gamma_0\in\Fpbar$ such that 
\begin{equation}
\label{eq:4}
\left|\beta^d-\alpha^d-\gamma_0\cdot\alpha\right|_v<|\alpha|_v.
\end{equation}
\end{claim}
The existence of $\gamma_0$ as in the conclusion of Claim~\ref{claim:key} is essential in the proof of Lemma~\ref{lem:4}. Furthermore, the argument used in the proof of Claim~\ref{claim:key} will also be useful in our further arguments for proving Lemma~\ref{lem:4}.

\begin{proof}[Proof of Claim~\ref{claim:key}.]
Let $K=\Fpbar(t,\alpha,\beta)$; then $K$ is a subfield of $L$. Moreover, $K$ is a global function field (of transcendence degree $1$) over $\Fpbar$ (it is a finite extension of $\Fpbar(t)$). We let $\mathcal{O}_{K,v}$ be the ring of $v$-adic integers in $K$; then $\mathcal{O}_{K,v}$ is a discrete valuation ring.  We let $\pi_v\in \mathcal{O}_{K,v}\subset L$ be a \emph{uniformizer} for the restriction of $|\cdot |_v$ on  $K$, i.e.,  
$$
|\pi_v|_v=\max\left\{|x|_v\colon |x|_v<1\text{ and }x\in K\right\}. 
$$ 
So,  there exists some positive integer $e$ with the property that both $(\beta^d-\alpha^d)\cdot \pi_v^e$ and $\alpha\cdot \pi_v^e$ are $v$-adic units in $K$ (note that $e>0$ since $|\beta^d-\alpha^d|_v=|\alpha|_v>1$). We let ${\rm red}_v:\mathcal{O}_{K,v}\lra \Fpbar$ be the reduction map at the place $v$; then we simply compute  
\begin{equation}
\label{eq:explanation_Fpbar}
\gamma_0:=\frac{{\rm red}_v\left(\left(\beta^d-\alpha^d\right)\cdot \pi_v^e\right)}{{\rm red}_v\left(\alpha\cdot \pi_v^e\right)}\in\Fpbar^*,  
\end{equation}
which satisfies the desired conclusion from Claim~\ref{claim:key}.
\end{proof}

So, we let $\gamma_0\in\Fpbar$ as in the conclusion of Claim~\ref{claim:key}.  
\begin{claim}
\label{claim:1}
With the above notation, there exists some $\gamma_1\in\Fpbar$ such that 
\begin{equation}
\label{eq:condition}
(\gamma_1+\gamma_0)^d=1\text{ and }\gamma_1^d\ne 1.
\end{equation}
\end{claim}

\begin{proof}[Proof of Claim~\ref{claim:1}.]
We argue by contradiction and therefore, we assume that for each $u\in\Fpbar$ such that $u^d=1$, we have that also $(u-\gamma_0)^d=1$. This means that the $d$-th roots of unity in $\Fpbar$ can be grouped in disjoint subsets of $p$ elements (note that $\gamma_0\ne 0$):
$$u,u-\gamma_0,u-2\gamma_0,\cdots,u-(p-1)\gamma_0.$$
However, this would mean that there are $p\cdot k$ solutions (for some positive integer $k$) for the equation $x^d=1$ in $\Fpbar$. This is a contradiction, because the equation $x^d=1$ has $s$ solutions in $\Fpbar$, where $d$ is written as $s\cdot p^\ell$ for some integer $\ell\ge 0$ and some positive integer $s$ coprime with $p$. The fact that $p$ does not divide $s$ shows that indeed one can find some $\gamma_1$ satisfying conditions~\eqref{eq:condition}.
\end{proof}

Then we consider (with $\gamma_1$ satisfying the conclusion of Claim~\ref{claim:1}) 
$$\lambda_1:=\gamma_1\cdot \alpha - \alpha^d;$$
a simple computation shows that 
$$
f^2_{\lambda_1}(\alpha)= f_{\lambda_1}(\gamma_1 \alpha)=(\gamma_1^d-1)\alpha^d +\gamma_1\alpha.
$$
Since $\gamma_1^d\ne 1$, we conclude that 
$$|f^2_{\lambda_1}(\alpha)|_v= |\alpha|_v^d>|\alpha|_v=|\lambda_1|_v^{1/d}$$ 
and thus, an application  of Lemma~\ref{lem:easy}~(ii) (coupled with equation~\eqref{eq:def_2}) yields
\begin{equation}
\label{eq:5}
\hhat_{v,\lambda_1}(\alpha)= \frac{\hhat_{v,\l_1}\left(f^2_{\l_1}(\alpha)\right)}{d^2} = \frac{\log|\alpha|_v^d}{d^2}=\frac{\log|\alpha|_v}{d}.
\end{equation}
On the other hand, we compute
\begin{equation}
\label{eq:3001}
f^2_{\lambda_1}(\beta)=f_{\lambda_1}(\beta^d-\alpha^d+\gamma_1\alpha)= (\beta^d-\alpha^d+\gamma_1\alpha)^d +\gamma_1\alpha - \alpha^d.
\end{equation}

\begin{claim}
\label{claim:key_2}
With our choice for $\gamma_0$ as in Claim~\ref{claim:key} and for $\gamma_1$ as in Claim~\ref{claim:1}, we have that 
\begin{equation}
\label{eq:tricky}
\left|(\beta^d-\alpha^d+\gamma_1\alpha)^d-\alpha^d\right|_v <|\alpha^d|_v.
\end{equation}
\end{claim}

\begin{proof}[Proof of Claim~\ref{claim:key_2}.]
In order to justify \eqref{eq:tricky}, we argue similarly as in our proof of Claim~\ref{claim:key}. So, letting as before, $\pi_v$ be a uniformizer of $K=\Fpbar(t,\alpha,\beta)$ at the place $v$, then for some positive integer $e$ and some $v$-adic units $u_1,u_2$ in $K$ (and therefore in $L$), we have 
\begin{equation}
\label{eq:3000}
\beta^d-\alpha^d=\pi_v^{-e}\cdot u_1\text{ and }\alpha=\pi_v^{-e}\cdot u_2;
\end{equation}
in particular, \eqref{eq:3000} yields $|\alpha^d|_v=\left(|\pi_v|_v^{-e}\right)^d=|\pi_v|_v^{-de}$. Then we obtain: 
$$
\left|(\beta^d-\alpha^d+\gamma_1\alpha)^d-\alpha^d\right|_v= \left| \pi_v^{-de}\cdot \left(u_1+\gamma_1u_2\right)^d-\pi_v^{-de}\cdot u_2^d\right|_v= |\pi_v|_v^{-de}\cdot \left| (u_1+\gamma_1u_2)^d-u_2^d\right|_v.
$$
So, in order to get inequality~\eqref{eq:tricky}, it suffices to prove that
\begin{equation}
\label{eq:3002}
\left| (u_1+\gamma_1u_2)^d-u_2^d\right|_v<1;
\end{equation}
furthermore, using that $u_1$ and $u_2$ are $v$-adic units,  inequality~\eqref{eq:3002} is equivalent with asking that 
\begin{equation}
\label{eq:3003}
\left| \left(\frac{u_1}{u_2} + \gamma_1\right)^d-1\right|_v<1.
\end{equation}
We re-write the left-hand side in \eqref{eq:3003} as
\begin{equation}
\label{eq:3004}
\left|\left(\frac{u_1-\gamma_0u_2}{u_2}+ (\gamma_0+\gamma_1)\right)^d-1\right|_v.
\end{equation}
Equation \eqref{eq:4} (see also \eqref{eq:explanation_Fpbar}) yields that 
\begin{equation}
\label{eq:3005}
\left|\frac{u_1-\gamma_0u_2}{u_2}\right|_v<1;
\end{equation}
so, coupling equations~\eqref{eq:3005}~and~\eqref{eq:3004},  
along with the fact that $(\gamma_0+\gamma_1)^d=1$, we obtain  inequality~\eqref{eq:3003} (which, in turn, delivers the desired inequality~\eqref{eq:tricky}).

This finishes our proof of Claim~\ref{claim:key_2}.
\end{proof}

Now, inequality~\eqref{eq:tricky} yields (see also equation~\eqref{eq:3001}) that  
\begin{equation}
\label{eq:6}
\left|f^2_{\lambda_1}(\beta)\right|_v<|\alpha|_v^d.
\end{equation}
We let $M:=\max\left\{|\alpha|_v, \left|f^2_{\lambda_1}(\beta)\right|_v\right\}$; inequality~\eqref{eq:6} yields 
\begin{equation}
\label{eq:5000}
M< |\alpha|_v^d.
\end{equation} 
Using inequality~\eqref{eq:6}, we obtain
\begin{equation}
\label{eq:14}
\left|f^3_{\lambda_1}(\beta)\right|_v = \left|\left(f^2_{\lambda_1}(\beta)\right)^d + \gamma_1 \alpha - \alpha^d\right|_v\le \max\left\{ \left|f^2_{\lambda_1}(\beta)\right|_v^d, |\alpha|_v^d\right\}=M^d.
\end{equation}
An easy induction (similar to deriving inequality~\eqref{eq:14}) shows then that for each $n\ge 2$, 
\begin{equation}
\label{eq:7}
\left|f^n_{\lambda_1}(\beta)\right|_v\le M^{d^{n-2}}.
\end{equation}
Inequality \eqref{eq:7} yields that $\hhat_{v,\lambda_1}(\beta)\le \frac{\log(M)}{d^2}$; then coupling this last inequality with equations~\eqref{eq:5}~and~\eqref{eq:5000}, we obtain that
$$\hhat_{v,\lambda_1}(\beta)<\hhat_{v,\lambda_1}(\alpha),$$
which contradicts our main hypothesis~\eqref{eq:hypothesis} from Proposition~\ref{prop:main}. This concludes our proof of Lemma~\ref{lem:4}.
\end{proof}


\subsection{Third step in the proof of Proposition~\ref{prop:main}}
\label{subsec:second}

We continue our analysis for $|\beta^d-\alpha^d|_v$ (for $v\in S$) with the goal of proving the inequality from~\eqref{eq:first_step}. This time, we need to split our proof depending whether $d=2$ or $d>2$. 
\begin{lemma}
\label{lem:5}
For each $v\in S$, we must have that 
\begin{itemize}
\item[(i)] if $d>2$, then $|\beta^d-\alpha^d|_v< 1$.
\item[(ii)] if $d=2$, then $|\beta^d-\alpha^d|_v\le 1$.
\end{itemize}
\end{lemma}

\begin{proof}[Proof of Lemma~\ref{lem:5}.]
We let $\lambda_2:=\alpha-\alpha^d$. Then clearly, $f_{\lambda_2}(\alpha)=\alpha$, which means that for any place $v$ (not just the ones from the set $S$), we have that
\begin{equation}
\label{eq:8}
\hhat_{v,\lambda_2}(\alpha)=0.
\end{equation}
From now on, we argue by contradiction and assume  that for some $v\in S$, we have 
\begin{itemize}
\item  $|\beta^d-\alpha^d|_v\ge 1$ if $d>2$. 
\item  $|\beta^d-\alpha^d|_v>1$ if $d=2$.
\end{itemize}
On the other hand, we know from Lemma~\ref{lem:4} that $|\beta^d-\alpha^d|_v<|\alpha|_v$.

Next, we write $d=p^\ell\cdot s$ for some integer $\ell\ge 0$ and some positive integer $s$ coprime with $p$. Furthermore, due to our hypothesis that $d\ne p^\ell$, then we must have that $s\ge 2$. We compute 
$$f^2_{\lambda_2}(\beta)=f_{\lambda_2}(\beta^d-\alpha^d+\alpha) =\left(\beta^d-\alpha^d+\alpha\right)^d+\alpha - \alpha^d.$$
Then we let $\epsilon:=\beta^d-\alpha^d$ and proceed as follows:
\begin{equation}
\label{eq:9}
f^2_{\lambda_2}(\beta)=(\epsilon+\alpha)^d-\alpha^d+\alpha= \left(\epsilon^{p^\ell}+\alpha^{p^\ell}\right)^s - \alpha^{sp^\ell}+\alpha.
\end{equation}
Then we expand 
$$\left(\epsilon^{p^\ell}+\alpha^{p^\ell}\right)^s= \alpha^{sp^\ell} + s\alpha^{(s-1)p^\ell}\epsilon^{p^\ell}+ \sum_{i=2}^s \binom{s}{i}\cdot \alpha^{(s-i)p^\ell}\epsilon^{ip^\ell}.$$
From our assumption that $|\beta^d-\alpha^d|_v\ge 1$ (for any $d\ge 2$) along with the conclusion of Lemma~\ref{lem:4}, we have that
\begin{equation}
\label{eq:12}
1\le |\epsilon|_v<|\alpha|_v\text{ (with the left inequality being strict if $d=2$)}.
\end{equation}
Since $p$ does not divide $s$, inequality~\eqref{eq:12} shows that 
\begin{equation}
\label{eq:11}
\left|s\alpha^{(s-1)p^\ell}\epsilon^{p^\ell}\right|_v = \left|\alpha^{(s-1)p^\ell}\epsilon^{p^\ell}\right|_v > \left|\binom{s}{i}\cdot \alpha^{(s-i)p^\ell}\epsilon^{ip^\ell}\right|_v
\end{equation}
for each $i=2,\dots,s$. Equation~\eqref{eq:11}  allows us to conclude that 
\begin{equation}
\label{eq:10}
\left|\left(\epsilon^{p^\ell}+\alpha^{p^\ell}\right)^s-\alpha^{sp^\ell}\right|_v = \left|\alpha^{(s-1)p^\ell}\epsilon^{p^\ell}\right|_v.
\end{equation}
Clearly, if $d>2$ then $(s-1)p^\ell\ge \max\{s-1, p^\ell\}\ge 2$; hence using again that $|\epsilon|_v\ge 1$ if $d>2$, we derive that 
\begin{equation}
\label{eq:16}
\left|\alpha^{(s-1)p^\ell}\epsilon^{p^\ell}\right|_v\ge |\alpha|_v^2>|\alpha|_v\text{ if }d>2.
\end{equation}
Furthermore, using our assumption that $|\epsilon|_v>1$ if $d=2$, then we also derive that
\begin{equation}
\label{eq:15}
\left|\alpha^{(s-1)p^\ell}\epsilon^{p^\ell}\right|_v= |\alpha\cdot \epsilon|_v>|\alpha|_v\text{ if }d=2,
\end{equation}
because then $p\ne 2$  and so, $\ell=0$ and $s=2$ if $d=2$. 
\begin{remark}
\label{rem:not_p_power}
We note that it is precisely in deriving inequalities \eqref{eq:16}~and~\eqref{eq:15} that we used the hypothesis that $d$ is not a power of the prime $p$, since this translates to the inequality $s\ge 2$, which is used in both of the above two inequalities (along with the fact that we cannot have $s=p=2$).
\end{remark}

Combining inequalities~\eqref{eq:16},~\eqref{eq:15},~\eqref{eq:10}~and~\eqref{eq:9} yields that
\begin{equation}
\label{eq:17}
\left|f^2_{\lambda_2}(\beta)\right|_v = \left|\alpha^{(s-1)p^\ell}\epsilon^{p^\ell}\right|_v>|\alpha|_v=|\lambda_2|_v^{1/d}.
\end{equation}
Inequality \eqref{eq:17} along with 
Lemma~\ref{lem:easy}~(ii) yields that
\begin{equation}
\label{eq:13}
\hhat_{v,\lambda_2}\left(f^2_{\lambda_2}(\beta)\right)= \log\left|\alpha^{(s-1)p^\ell}\epsilon^{p^\ell}\right|_v>0.
\end{equation}
Finally, using equations~\eqref{eq:13}~and~\eqref{eq:def_2}, we conclude that $\hhat_{v,\lambda_2}(\beta)>0$. Coupled with equation~\eqref{eq:8}, this contradicts the main hypothesis~\eqref{eq:hypothesis} that $\hhat_{v,\lambda_2}(\alpha)=\hhat_{v,\lambda_2}(\beta)$. This  concludes our proof of Lemma~\ref{lem:5}.  
\end{proof}


\subsection{Final step in the proof of Proposition~\ref{prop:main}}
\label{subsec:third}

Now we can finish our proof of Proposition~\ref{prop:main}. Once again, we split our analysis into two cases: $d>2$, respectively $d=2$.

\begin{lemma}
\label{lem:d>2}
If $d>2$, then the conclusion in Proposition~\ref{prop:main} holds.
\end{lemma}

\begin{proof}[Proof of Lemma~\ref{lem:d>2}.]
By definition of the set $S$ (see \eqref{eq:2112}), we have that if $|\beta^d-\alpha^d|_v>1$, then we must have that $v\in S$. On the other hand, Lemma~\ref{lem:5} yields that $|\beta^d-\alpha^d|_v<1$ if $v\in S$. Hence, $\beta^d-\alpha^d$ is integral at all places and furthermore, for the places $v\in S$ (note that $S$ is nonempty due to our assumption that not both $\alpha$ and $\beta$ are in $\Fpbar$), we have that $|\beta^d-\alpha^d|_v<1$; this contradicts the product formula~\eqref{eq:product_formula}, unless $\alpha^d-\beta^d=0$, which is precisely the desired conclusion from Proposition~\ref{prop:main}.
\end{proof}

\begin{lemma}
\label{lem:d=2}
If $d=2$, the conclusion in Proposition~\ref{prop:main} must hold.
\end{lemma}

\begin{proof}
The same argument as in the proof of Lemma~\ref{lem:d>2} yields that $|\beta^2-\alpha^2|_v\le 1$ for each place $v\in S$. Since we already know that $|\beta^2-\alpha^2|_v\le 1$ for $v\notin S$, then \eqref{eq:finite_field} yields that $\epsilon=\beta^2-\alpha^2\in\Fpbar$.  

Furthermore, we note that $p\ne 2$ since we know that $d=2$ is not a power of the prime $p$ (according to the hypothesis in Proposition~\ref{prop:main}).

Again we work with $\lambda_2=\alpha-\alpha^2$ as in Lemma~\ref{lem:5}. We compute (with $\epsilon=\beta^2-\alpha^2$)
\begin{equation}
\label{eq:20}
f_{\lambda_2}(\beta)=\beta^2+\alpha-\alpha^2=\alpha+\epsilon\text{ and }f^2_{\lambda_2}(\beta)=(2\epsilon+1)\alpha +\epsilon^2
\end{equation}
and then
\begin{equation}
\label{eq:21}
f^3_{\lambda_2}(\beta)=(4\epsilon^2+4\epsilon)\alpha^2+ (4\epsilon^3+2\epsilon^2+1)\alpha+\epsilon^4.
\end{equation}
Now, if $4\epsilon^2+4\epsilon\ne 0$, then we get (note that $\epsilon\in\Fpbar$):  
\begin{equation}
\label{eq:18}
\left|f^3_{\lambda_2}(\beta)\right|_v=|\alpha|_v^2\text{ for $v\in S$.}
\end{equation}
Using inequality~\eqref{eq:18} along with Lemma~\ref{lem:easy}~(ii) yields that for $v\in S$, we would have that $\hhat_{v,\lambda_2}(f^3_{\lambda_2}(\beta))>0$ and therefore, we would also have that $\hhat_{v,\lambda_2}(\beta)>0$, which contradicts that $\hhat_{v,\lambda_2}(\alpha)=0$ because $\alpha$ is a  fixed point under the action of $f_{\lambda_2}$. So, we must have instead that 
\begin{equation}
\label{eq:19}
4\epsilon^2+4\epsilon=0.
\end{equation}
Since $p\ne 2$, then equation~\eqref{eq:19} yields either $\epsilon=0$ (which provides the desired conclusion $\alpha^2=\beta^2$), or we must have $\epsilon=-1$. However, this last possibility would  provide a contradiction since then we can run the same argument with $\alpha$ and $\beta$ reversed and we would have gotten that 
$$\alpha^2-\beta^2=-1,$$
thus contradicting $\beta^2-\alpha^2=-1$. More precisely, we could consider next $$\lambda_3:=\beta-\beta^2$$
and impose the condition that $\hhat_{v,\lambda_3}(\alpha)=0$ for each $v\in S$. Letting $\mu:=-\epsilon=\alpha^2-\beta^2$ and then arguing along the same lines as in the derivation of equations~\eqref{eq:20},~\eqref{eq:21},~\eqref{eq:18}~and~\eqref{eq:19}, we would obtain that either $\mu=0$ (which is the desired conclusion), or that $\mu=-1$. Since we cannot have that 
$$\epsilon=-1\text{ and }\mu=-\epsilon=-1$$
because the characteristic $p$ of our field is not $2$ (because $d=2$ is not a power of the characteristic), we conclude that also when $d=2$, we must have that $\alpha^2=\beta^2$.

This concludes our proof of Lemma~\ref{lem:d=2}.
\end{proof}

Combining Lemmas~\ref{lem:d>2} and \ref{lem:d=2}, we obtain the desired conclusion in Proposition~\ref{prop:main}.


\section{Proof of our main results}
\label{sec:proofs}

In this Section, we finish our proof for Theorem~\ref{thm:main}. We actually prove a more general result.

\begin{theorem}
\label{thm:main_2}
Let $L_0=\Fpbar\left(t,t^{1/p},t^{1/p^2},\cdots,t^{1/p^n},\cdots\right)$ and let $L$ be a finite extension of $L_0$. 
Let $d\ge 2$ be an integer and let $\alpha,\beta\in L$. We consider the family of polynomials 
$$f_\lambda(x):=x^d+\lambda\text{ parameterized by $\lambda\in \Lbar$.}$$ 
Then there exists an infinite sequence $\{\lambda_n\}_{n\ge 1}$ in $\Lbar$ with the property that
\begin{equation}
\label{eq:weaker_condition}
\lim_{n\to\infty}\hhat_{f_{\lambda_n}}(\alpha)=\lim_{n\to\infty} \hhat_{f_{\lambda_n}}(\beta)=0,
\end{equation}
if and only if at least one of the following statements holds:
\begin{itemize}
\item[(1)] $\alpha,\beta\in\Fpbar $.
\item[(2)] $d=p^\ell$ for some positive integer $\ell$ and $\beta-\alpha\in\Fpbar$.
\item[(3)] $\alpha^d=\beta^d$.
\end{itemize}
Moreover, if either one of the conditions~(1)-(3) holds, then for each $\lambda\in\Lbar$, we have that $\alpha$ is preperiodic under the action of $f_\lambda$ if and only if $\beta$ is preperiodic under the action of $f_\lambda$.  
\end{theorem}

We start by proving Theorem~\ref{thm:main} assuming that Theorem~\ref{thm:main_2} holds; this is done in Subsection~\ref{subsec:trdeg}. Then we will prove Theorem~\ref{thm:main_2} by splitting our argument into two steps in Subsections~\ref{subsec:moreover}~and~respectively,~\ref{subsec:conclusion_proof}.


\subsection{Proof of Theorem~\ref{thm:main} assuming Theorem~\ref{thm:main_2} holds}
\label{subsec:trdeg}

The next Proposition shows that in Theorem~\ref{thm:main} we may assume $\trdeg_{\Fp}L=1$. 
\begin{proposition}
\label{prop:trdeg}
Let $d\ge 2$ be an integer, let $L$ be a field of characteristic $p>0$, and let $\alpha,\beta\in L$. If there exists $\lambda_1\in\Lbar$ such that both $\alpha$ and $\beta$ are preperiodic for the polynomial $f_{\lambda_1}(x)=x^d+\lambda_1$, then $\trdeg_{\Fp}\Fp(\alpha,\beta)\le 1$.
\end{proposition}

\begin{proof}
We recall the notation from Subsection~\ref{subsec:prep} (see Lemma~\ref{lem:degree in l}) that for each $\gamma\in L$, we have 
\begin{equation}
\label{eq:22}
P_{n,\gamma}(\lambda):=f^n_\lambda(\gamma),
\end{equation}
which is a monic polynomial of degree $d^{n-1}$ (in $\lambda$). Furthermore, the constant term is $P_{n,\gamma}(0)=\gamma^{d^n}$. An easy induction yields that each coefficient of $\lambda^i$ in $P_{n,\gamma}(\l)$ for $i=1,\dots, d^{n-1}-1$ is itself a polynomial \emph{in $\gamma$}, i.e., 
\begin{equation}
\label{eq:23}
P_{n,\gamma}(\lambda)=\lambda^{d^{n-1}}+ \sum_{i=1}^{d^{n-1}-1} c_{n,i}(\gamma)\cdot \lambda^i + \gamma^{d^n},
\end{equation}
with each $c_{n,i}\in \Fp[x]$ being a polynomial of degree less than $d^n$. Therefore, imposing the condition that $\alpha$ is a preperiodic point under the action of some $f_{\lambda_1}$ yields an equation of the form:
$$P_{n,\alpha}(\lambda_1)=P_{m,\alpha}(\lambda_1)\text{ for some }0\le m<n.$$
Using equation~\eqref{eq:23} (along with the information about the degrees of each corresponding polynomials $c_{m,i}$ and $c_{n,j}$), we obtain that $\alpha\in \overline{\Fp(\lambda_1)}$. A similar reasoning, using this time that $\beta$ is preperiodic under the action of $f_{\lambda_1}$, yields that also $\beta\in\overline{\Fp(\lambda_1)}$. Hence, we conclude that 
$$\trdeg_{\Fp}\Fp(\alpha,\beta)\le \trdeg_{\Fp}\overline{\Fp(\lambda_1)}\le 1,$$
as desired for the conclusion of Proposition~\ref{prop:trdeg}. 
\end{proof}

\begin{proof}[Proof of Theorem~\ref{thm:main} as a consequence of Theorem~\ref{thm:main_2}.]
First, we note that Theorem~\ref{thm:main} is left unchanged if we replace $L$ by \emph{any} field extension of $\Fpbar(\alpha,\beta)$. 

Second, Proposition~\ref{prop:trdeg} allows us to assume that $\trdeg_{\Fp}L=1$ in Theorem~\ref{thm:main}; i.e. $\alpha$ and $\beta$ live in a fixed algebraic closure of $\Fp(t)$. 

So, from now on, we take $L_0$ be the perfect closure of the rational function field in one variable over $\Fpbar$, i.e.,
\begin{equation}
\label{eq:24}
L_0=\Fpbar(t)^{\rm per}=\Fpbar\left(t, t^{\frac{1}{p}}, t^{\frac{1}{p^2}},\cdots, t^{\frac{1}{p^n}},\cdots \right)
\end{equation}
and we let $L$ be a finite extension of $L_0$ containing both $\alpha$ and $\beta$.

Now, assume that we have infinitely many $\lambda_n\in\Lbar$ with the property that both $\alpha$ and $\beta$ are preperiodic under the action of $f_{\lambda_n}$. Then (according to Remark~\ref{rem:clearly}), we have that $\hhat_{f_{\lambda_n}}(\alpha)= \hhat_{f_{\lambda_n}}(\beta)=0$ for each $n\ge 1$; hence, the direct implication from Theorem~\ref{thm:main_2} yields that at least one of the conditions~(1)-(3) are met. 

Conversely, assuming that at least one of the conditions~(1)-(3) are met, then the ``moreover'' statement from Theorem~\ref{thm:main_2} yields that for each $\lambda\in\Lbar$, $\alpha$ is preperiodic under the action of $f_\l$ if and only if $\beta$ is preperiodic under the action of $f_\l$. So, in order to prove the converse statement in Theorem~\ref{thm:main}, it suffices to establish the following fact.

\begin{proposition}
\label{prop:infinitely}
Let $L$ be an arbitrary field of characteristic $p$ and let $\gamma\in L$. Then there exist infinitely many $\lambda\in \Lbar$ such that $\gamma$ is preperiodic under the action of $f_\lambda$.
\end{proposition}

If $L$ were $\mathbb{C}$, then the conclusion from Proposition~\ref{prop:infinitely} follows from a more general result established in \cite{Laura-2}; however, since $L$ is a field of characteristic $p$, once again we require a different proof.

\begin{proof}[Proof of Proposition~\ref{prop:infinitely}.]
If $\gamma\in\Fpbar$, then the statement is obvious because then $\gamma$ is preperiodic under $f_\lambda$ for each $\lambda\in\Fpbar$. So, from now on, we assume $\gamma\in L\setminus\Fpbar$. The desired conclusion in Proposition~\ref{prop:infinitely} follows from the next Lemma, which provides a more refined conclusion.

\begin{lemma}
\label{lem:infinitely}
Assume $\gamma\notin\Fpbar$. Then there exist infinitely many $\lambda\in\Lbar$ with the property that there exists some prime number $q$ such that $f^q_\lambda(\gamma)=\gamma$.
\end{lemma}

\begin{proof}[Proof of Lemma~\ref{lem:infinitely}.]
We argue by contradiction and so, assume the set 
$$\mathcal{P}:=\left\{\lambda\in\Lbar\colon \text{ there exists a prime $q$ such that }f^q_\lambda(\gamma)=\gamma\right\}$$
is finite. In particular, this means that there exists a positive integer $M$ with the property that for each prime $q>M$ and for each $\lambda\in\Lbar$ such that 
\begin{equation}
\label{eq:27}
f^q_\lambda(\gamma)=\gamma, 
\end{equation}
there exists a prime $q_0<M$ (with $q_0$ depending on $\lambda$, of course) such that 
\begin{equation}
\label{eq:28}
f^{q_0}_\lambda(\gamma)=\gamma. 
\end{equation}
However, since $q$ and $q_0$ are distinct primes, then equations~\eqref{eq:27}~and~\eqref{eq:28} yield that $f_\lambda(\gamma)=\gamma$, i.e., $\lambda=\gamma-\gamma^d$. Hence, letting $P_{q,\gamma}(\lambda):=f^q_\lambda(\gamma)$ as before (see equation~\eqref{eq:22}), the only solution $\lambda\in \Lbar$ to the equation $P_{q,\gamma}(\lambda)=\gamma$ is $\lambda_0:=\gamma-\gamma^d$. Now,   using the shape of the polynomial $P_{q,\gamma}(\lambda)$ (see equation~\eqref{eq:23}), we conclude that
\begin{equation}
\label{eq:29}
P_{q,\gamma}(\lambda)=\left(\lambda-\lambda_0\right)^{d^{q-1}}.
\end{equation}
In particular, this means that the constant term in the polynomial $P_{q,\gamma}$ must be $(\gamma^d-\gamma)^{d^{q-1}}$. On the other hand, we know that the constant term in the polynomial $P_{q,\gamma}$ is $\gamma^{d^q}$; this leads to the equation
\begin{equation}
\label{eq:30}
\left(\gamma^d-\gamma\right)^{d^{q-1}}=\gamma^{d^q}.
\end{equation} 
Any solution $\gamma$ to equation~\eqref{eq:30} must live in $\Fpbar$, thus contradicting the hypotheses of Lemma~\ref{lem:infinitely}. Therefore, indeed, the set $\mathcal{P}$ must be infinite, as claimed in the conclusion of Lemma~\ref{lem:infinitely}.
\end{proof}
Lemma~\ref{lem:infinitely} shows that also when $\gamma\notin\Fpbar$, there exist infinitely many $\lambda\in\Lbar$ such that $\gamma$ is preperiodic under the action of $f_\lambda$. This concludes our proof for Proposition~\ref{prop:infinitely}.
\end{proof}
So, Proposition~\ref{prop:infinitely} shows that there exist infinitely many $\l\in\Lbar$ such that $\alpha$ is preperiodic under the action of $f_\l$ and therefore (according to the ``moreover'' claim in Theorem~\ref{thm:main_2}), also $\beta$ is preperiodic under the action of $f_\l$. Hence, this establishes the converse statement in Theorem~\ref{thm:main}.

This concludes our proof of Theorem~\ref{thm:main}, assuming that Theorem~\ref{thm:main_2} holds.
\end{proof}


\subsection{Strategy for proving Theorem~\ref{thm:main_2}}  
We split our proof of Theorem~\ref{thm:main_2} into the remaining two Subsections of the current Section~\ref{sec:proofs}. In particular, we prove the ``moreover'' claim from Theorem~\ref{thm:main_2} in Subsection~\ref{subsec:moreover} and then we finish the proof of Theorem~\ref{thm:main_2} in Subsection~\ref{subsec:conclusion_proof}. 

So, from now on, we work with the notation and the assumptions from Theorem~\ref{thm:main_2}.


\subsection{Proof of the ``moreover'' claim from Theorem~\ref{thm:main_2}}
\label{subsec:moreover}

In this Subsection, we show that if either one of conditions~(1)-(3) from the conclusion of Theorem~\ref{thm:main_2} holds, then for each $\lambda\in\Lbar$, we have that $\alpha$ is preperiodic under the action of $f_\lambda$ if and only if $\beta$ is preperiodic under the action of $f_\lambda$.

We argue case by case, as follows.
\begin{itemize}
\item[(1)] For any $\gamma\in\Fpbar$, using equations~\eqref{eq:22}~and~\eqref{eq:23}, we get that for each $\lambda\in\Lbar$, we have that $\gamma$ is preperiodic under the action of $f_\lambda$ if and only if $\lambda\in\Fpbar$. Therefore, if both $\alpha$ and $\beta$ live in $\Fpbar$, we have that for each $\lambda\in\Lbar$, $\alpha$ is preperiodic under the action of $f_\lambda$ if and only if $\beta$ is preperiodic under the action of $f_\lambda$.
\item[(2)] Now, assume $d=p^\ell$ for some positive integer $\ell$. Then a simple induction on $n$ shows that for each $\gamma\in L$, we have (see equations~\eqref{eq:22}~and~\eqref{eq:23}) that
\begin{equation}
\label{eq:25}
P_{n,\gamma}(\lambda)=f^n_\lambda(\gamma)= \gamma^{p^{n\ell}} + \sum_{i=0}^{n-1} \lambda^{p^{i\ell}}.
\end{equation}
Now, assume also that $\alpha-\beta\in \Fpbar$; more precisely, we assume $\nu:=\alpha-\beta\in \mathbb{F}_{p^{\ell m}}$ for some positive integer $m$. Then equation~\eqref{eq:25} yields that for each $n\ge 1$, we have
\begin{equation}
\label{eq:26}
f^n_\lambda(\alpha)-f^n_\lambda(\beta) =\nu^{p^{n\ell}}; 
\end{equation}  
moreover, the elements $\left\{\nu^{p^{n\ell}}\right\}_{n\ge 1}$ cycle among the values 
$$\nu,\nu^{p^\ell},\nu^{p^{2\ell}},\cdots, \nu^{p^{(m-1)\ell}}\text{ (since $\nu^{p^{m\ell}}=\nu$).}$$
Therefore, $\alpha$ is preperiodic under the action of $f_\lambda$ if and only if $\beta$ is preperiodic under the action of $f_\lambda$, as desired in the ``moreover'' claim from the conclusion of Theorem~\ref{thm:main_2}.
\item[(3)] Finally, if $\alpha^d=\beta^d$, we note that for each $\lambda\in\Lbar$, we have that $f_\lambda(\alpha)=f_\lambda(\beta)$ and therefore, $\alpha$ is preperiodic under the action of $f_\lambda$ if and only if $\beta$ is preperiodic under the action of $f_\lambda$. 
\end{itemize}

This concludes our proof that whenever one of the conditions~(1)-(3) from the conclusion of Theorem~\ref{thm:main_2} holds, then for each $\lambda\in\Lbar$, we have that $\alpha$ is preperiodic for $f_\lambda$ if and only if $\beta$ is preperiodic under the action of $f_\lambda$. Furthermore, according to Proposition~\ref{prop:infinitely}, we know there exist infinitely many $\l\in\Lbar$ such that $\alpha$ (and therefore, also $\beta$) is preperiodic under the action of $f_\l$. Therefore,  either one of the conditions~(1)-(3) from Theorem~\ref{thm:main_2} yields the existence of infinitely many $\l_n\in\Lbar$ such that both $\alpha$ and $\beta$ are preperiodic under the action of $f_{\l_n}$. Clearly (see also Remark~\ref{rem:clearly}), for each such $\l_n\in\Lbar$, we have 
\begin{equation}
\label{eq:201}
\hhat_{f_{\l_n}}(\alpha)=\hhat_{f_{\l_n}}(\beta)=0.
\end{equation}
In particular, equation~\eqref{eq:201} allows us to establish the converse statement in Theorem~\ref{thm:main_2}. So, it remains to prove the direct implication from the statement of Theorem~\ref{thm:main_2}; this is done in Subsection~\ref{subsec:conclusion_proof}.


\subsection{Conclusion of the proof of Theorem~\ref{thm:main_2}}
\label{subsec:conclusion_proof}

In this Subsection, we prove the last remaining statement from Theorem~\ref{thm:main_2}, i.e., that if there exist infinitely many $\lambda_n\in\Lbar$ such that $\lim_{n\to\infty}\hhat_{f_{\l_n}}(\alpha) = \lim_{n\to\infty}\hhat_{f_{\l_n}}(\beta)=0$, then at least one of the conditions~(1)-(3) from the conclusion of Theorem~\ref{thm:main_2} must hold.

Theorem~\ref{thm:iff} shows that condition~\eqref{eq:weaker_condition} yields that for each $\lambda\in\Lbar$ and for each place $v\in\Omega=\Omega_L$, we have that 
\begin{equation}
\label{eq:6001}
\hhat_{v,\lambda}(\alpha)=\hhat_{v,\lambda}(\beta),
\end{equation} 
i.e., hypothesis~\eqref{eq:hypothesis} from Proposition~\ref{prop:main} is met.  Then Proposition~\ref{prop:main} yields that 
\begin{itemize}
\item[(1)] either both $\alpha$ and $\beta$ live in $\Fpbar$; 
\item[(2)] or $d=p^\ell$ for some positive integer $\ell$; 
\item[(3)] or $\alpha^d=\beta^d$. 
\end{itemize}
So, it remains to prove that when $d=p^\ell$ then we \emph{must} have that also $\alpha-\beta\in\Fpbar$. Indeed, we will see that the existence of a \emph{single}  $\l_1\in\Lbar$ such that both $\alpha$ and $\beta$ are preperiodic under the action of 
\begin{equation}
\label{eq:6003}
f_{\l_1}(x)=x^d+\l_1=x^{p^\ell}+\l_1
\end{equation}
would give that $\alpha-\beta\in\Fpbar$. Now, the existence of such $\l_1\in\Lbar$ comes from the fact that once $\alpha$ is preperiodic under the action of some $f_{\l_1}$ (and there are infinitely many such parameters $\l_1\in\Lbar$ according to Proposition~\ref{prop:infinitely}), then actually $\alpha$ is preperiodic under the action of $f_{\sigma(\l_1)}$ for any $\sigma\in\gal(L^{\rm sep}/L)$ (since $\sigma(f^n_{\l_1}(\alpha))=f^n_{\sigma(\l_1)}(\alpha)$ for any $n$, because $\alpha\in L$ is fixed by $\sigma$). But then (see Remark~\ref{rem:clearly}), we have
\begin{equation}
\label{eq:6000}
\hhat_{v,\sigma(\l_1)}(\alpha)=0\text{ for each }v\in\Omega\text{ and for each }\sigma\in\gal(L^{\rm sep}/L).
\end{equation}
Equations~\eqref{eq:6000} and \eqref{eq:6001} yield that $\hhat_{v,\sigma(\l_1)}(\beta)=0$ for each $v\in\Omega$ and for each $\sigma\in\gal(L^{\rm sep}/L)$. Using equations~\eqref{def ade hig},~\eqref{eq:G}~and~\eqref{eq:height_adelic} along with Lemma~\ref{lem:global_height}, we conclude that $\hhat_{f_{\l_1}}(\beta)=0$. But then, as noted in Remark~\ref{rem:clearly}, also the converse holds: we must have that $\beta$ must be preperiodic under the action of $f_{\l_1}$.

So, knowing that both $\alpha$ and $\beta$ are preperiodic under the action of $f_{\l_1}$,  we argue as in Subsection~\ref{subsec:moreover} (see equations~\eqref{eq:25},~\eqref{eq:26}~and~\eqref{eq:6003}) and writing $\nu=\alpha-\beta$, then we get that the sequence $\left\{\nu^{p^{\ell n}}\right\}_{n\ge 1}$ must consist of finitely many distinct elements (because its elements are the differences of the elements in the orbits of $\alpha$ and respectively, of  $\beta$ under the action of $f_{\l_1}$). This can only happen if $\nu\in\Fpbar$, as desired. 

This concludes our proof of Theorem~\ref{thm:main_2}.


\end{document}